\documentclass[11pt,reqno]{amsart}
\usepackage{amsfonts,latexsym,rawfonts,amsmath,amssymb,amsthm,a4wide}
\usepackage{graphicx}
\usepackage[plainpages=false]{hyperref}
\numberwithin{equation}{section}

\renewcommand{\Im}{\mathrm{Im}}
\renewcommand{\Re}{\mathrm{Re}}
\newcommand{\pd}[2]{\frac {\partial #1}{\partial #2}}
\newcommand{\al}{\alpha}
\newcommand{\bb}{\beta}
\newcommand{\la}{\lambda}
\newcommand{\La}{\Lambda}
\newcommand{\oo}{\omega}

\newcommand{\Na}{\nabla}

\newcommand{\ee}{\epsilon}

\newcommand{\te}{\theta}

\newcommand{\beq}{\begin{equation}}
\newcommand{\eeq}{\end{equation}}
\newcommand{\beqs}{\begin{eqnarray*}}
\newcommand{\eeqs}{\end{eqnarray*}}
\newcommand{\beqn}{\begin{eqnarray}}
\newcommand{\eeqn}{\end{eqnarray}}
\newcommand{\beqa}{\begin{array}}
\newcommand{\eeqa}{\end{array}}

\def\td{\tilde}
\def\p{\partial}

\def\RR{{\mathbb R}}

\def\pbp{\frac {\sqrt{-1}}2 \partial\bar\partial}

\def\cH{{\mathcal H}}

\def\Aut{{\rm Aut}}
\newtheorem{prop}{Proposition}[section]
\newtheorem{theo}[prop]{Theorem}
\newtheorem{lem}[prop]{Lemma}

\newtheorem{cor}[prop]{Corollary}

\title[]{A criterion for the properness of the $K$-energy in a general K\"ahler class (II)}

\author{Haozhao Li$^1$  }
\address{ Department of Mathematics, University of Science and Technology
of China, Hefei, 230026, Anhui province, China and Wu Wen-Tsun Key
Laboratory of Mathematics, USTC, Chinese Academy of Sciences, Hefei
230026, Anhui,  China} \email{hzli@ustc.edu.cn}

\author{Yalong Shi$^2$}
\address{
Department of Mathematics and Institute of Mathematical Science,
Nanjing University, Nanjing, 210093, Jiangsu province, China}
\email{shiyl@nju.edu.cn}

\thanks{$^1$Research
partially supported by NSFC grant  No. 11131007.}

\thanks{$^2$Research partially supported by NSFC grants No. 11101206.}

\begin{document}

\bibliographystyle{plain}

\date{}

\maketitle

\begin{abstract} In this paper, we give a result on the properness
of the $K$-energy, which answers a question of Song-Weinkove
\cite{[SW]} in any dimensions. Moreover, we extend our previous
result on the properness of $K$-energy in \cite{[LSY]} to the
case of modified $K$-energy associated to extremal K\"ahler metrics.

\end{abstract}

\tableofcontents
\section{Introduction}

This paper is a continuation of our previous work \cite{[LSY]}. In
\cite{[LSY]}, we give a criterion for the properness
 of the $K$-energy in a general K\"ahler class of a compact K\"ahler
 manifold by using Song-Weinkove's result on $J$-flow in \cite{[SW]}, which
 extends the works of Chen \cite{[Chen1]}, Song-Weinkove
 \cite{[SW]} and Fang-Lai-Song-Weinkove \cite{[FLSW]}. In
 \cite{[SW]}, Song-Weinkove showed that the
$K$-energy is proper  on a K\"ahler class $[\chi_0]$ of a
$n$-dimensional K\"ahler manifold $M$ with $c_1(M)<0$ whenever there are  K\"ahler metrics $\oo\in -\pi c_1(M)$ and $\chi'\in [\chi_0]$ such
that \beq \Big(-n\frac { \pi c_1(M)\cdot
[\chi_0]^{n-1}}{[\chi_0]^n}\chi'-(n-1)\oo\Big) \wedge
\chi'^{n-2}>0. \label{eq:B1} \eeq Moreover, Song-Weinkove asked whether the $K$-energy is bounded from below if the inequality
(\ref{eq:B1}) is not strict (Remark 4.2 of \cite{[SW]}). In  \cite{[FLSW]}
 Fang-Lai-Song-Weinkove  studied the $J$-flow on the boundary of the K\"ahler cone and
 gave an affirmative answer in complex dimension 2. In \cite{[LSY]},
 we give a partial answer to this question, which says
 that the $K$-energy is proper if $c_1(M)<0$ and the K\"ahler class
 $[\chi_0]$ satisfies
 \beq
 -n\frac { c_1(M)\cdot
[\chi_0]^{n-1}}{[\chi_0]^n}[\chi_0]+(n-1)c_1(M) \geq 0.  \nonumber
\eeq

The first main result in this paper is the following theorem, which
answers the question of Song-Weinkove in any dimensions.

\begin{theo}\label{theo:Bmain1}Let $M$ be a $n$-dimensional compact K\"ahler manifold
with $c_1(M)<0$. If the K\"ahler class $[\chi_0]$ satisfies the
property that  there are two K\"ahler metrics $\chi'\in [\chi_0]$
and $\oo\in -\pi c_1(M)$ such that \beq \Big(-n\frac {\pi
c_1(M)\cdot [\chi_0]^{n-1}}{[\chi_0]^n}\chi'-(n-1)\oo\Big) \wedge
\chi'^{n-2}\geq 0, \label{eq:109} \eeq then the $K$-energy is proper
on the K\"ahler class $[\chi_0]$.
\end{theo}

Our second main result is to extend \cite{[LSY]} to the case of
extremal K\"ahler metrics. To state our main results, we recall
Tian's $\alpha$-invariant for a K\"ahler class $[\chi_0]$:
$$\al_M([\chi_0])=\sup\Big\{\al>0\;\Big|\;\exists \,C>0,\;
 \int_M\; e^{-\al(\varphi-\sup\varphi)}\chi_0^n\leq C,
 \quad \forall\;\varphi\in \cH(M,  \chi_0 )\Big\},$$
 where $\cH(M,  \chi_0 )$ denotes the space of K\"ahler potentials
 with respect to the metric $\chi_0$.
For any compact subgroup $G$ of $\Aut(M)$, and a $G$-invariant
K\"ahler class $[\chi_0]$, we can similarly define the
$\alpha_{M,G}$ invariant by using $G$-invariant potentials in the
definition.

\begin{theo}\label{theo:Amain1}Let $M$ be a $n$-dimensional compact K\"ahler manifold and
$X$ an extremal vector field of the K\"ahler class $[\chi_0]$ with potential function $\theta_X:=\theta_X(\chi_0)$. Assume $\Im X$ generates a compact group of holomorphic automorphisms\footnote{If $[\chi_0]=c_1(M)$, then this is always true, see Theorem F of \cite{[FM]}.}, and $L_{  \Im X }\chi_0=0$.
If the K\"ahler class $[\chi_0]$ satisfies the following conditions
for some constant $\ee:$
\begin{enumerate}
  \item[(1)] $0\leq \ee<\frac {n+1}{n}\al_M([\chi_0]), $
  \item[(2)] $\pi c_1(M)<(\ee+\min\theta_X) [\chi_0],$
  \item[(3)]   \beq
 \Big(-n\frac {\pi c_1(M)\cdot
[\chi_0]^{n-1}}{[\chi_0]^n}+\min_M\te_X+\ee\Big)[\chi_0]+(n-1)\pi
c_1(M)  >0, \nonumber \eeq
\end{enumerate}
then the modified $K$-energy is proper on $\cH_X(M, \chi_0),$ where $\cH_X(M, \chi_0)$ is the subspace of $\cH(M, \chi_0)$ with the extra condition $\Im X (\varphi)=0$.  If instead of (1), we assume  $[\chi_0]$ is $G$-invariant for a compact subgroup $G$ of $Aut(M)$ ,  and $0\leq \ee<\frac {n+1}{n}\al_{M,G}([\chi_0])$, then  the modified  $K$-energy is proper on the space of $G$-invariant potentials.
\end{theo}

For the definitions of extremal vector field and modified $K$-energy, see section \ref{sec:Amain1}. Note that the $\min \theta_X$ here is actually an invariant of the K\"ahler class according to \cite{[M]} and Appendix of \cite{[ZhZh]}. We can also replace the condition (3) of Theorem \ref{theo:Amain1} by some weaker assumptions as in Theorem \ref{theo:Bmain1}, however we prefer this version since (3) is easier to check.  The proof of Theorem \ref{theo:Amain1} relies on the study of the
modified $J$-flow, which is an extremal version of the usual $J$
flow defined by Donaldson \cite{[Don1]} and Chen \cite{[Chen1]}.
Here we modify the proof of Song-Weinkove \cite{[SW]} to get the
existence of critical metrics of the modified $J$ functional and
then we apply the argument in \cite{[LSY]} to get the properness of
the modified $K$-energy.

 In a recent interesting paper \cite{[Der2]}, Dervan
gives a different sufficient condition on the properness of the $K$
energy on a general K\"ahler class by direct analyzing the
expression of the $K$ energy, which gives better results in some
examples (cf. \cite{[LSY]}\cite{[Der2]}). While Dervan's condition is useful mainly when $M$ is Fano,  our theorem applies on more general manifolds. Whether one can improve both results is still an interesting problem.

In section \ref{sec:Bmain1}, we prove Theorem \ref{theo:Bmain1}, which is a strengthen of the Main Theorem of \cite{[LSY]}. Then in section \ref{sec:Amain1}, we study the modified $J$-flow and prove Theorem \ref{theo:Amain1}.\\

\section{Proof of Theorem \ref{theo:Bmain1}} \label{sec:Bmain1}

In this section we prove Theorem \ref{theo:Bmain1}. Here we use the
notations in our previous work \cite{[LSY]}.

\begin{proof}[Proof of Theorem \ref{theo:Bmain1}]
By the assumption (\ref{eq:109}),  for sufficiently small $\ee>0$ we
have
    \beq
 \Big( (nc+\ee)\chi'-(n-1)\oo\Big )\wedge \chi'^{n-2}  >0,
 \label{eq:A0001}
\eeq where \beq c=\frac{-\pi c_1(M)\cdot
[\chi_0]^{n-1}}{[\chi_0]^n}.\nonumber \eeq We can write
$(\ref{eq:A0001})$ as \beq
 \Big( n(c+\ee)\chi'-(n-1)(\oo+\ee \chi')\Big )\wedge \chi'^{n-2}  >0,
 \label{eq:A002}
\eeq Since $\oo$ and $\chi'$ are K\"ahler metrics, we have
$\oo+\ee\chi'>0$.  by Song-Weinkove's result (cf. Theorem 1.1 in
\cite{[SW]}) there exists a K\"ahler metric $\chi\in [\chi_0]$ such
that
$$(\oo+\ee\chi')\wedge \chi^{n-1}=(c+\ee)\chi^n.$$
Thus, the functional $\hat J_{\oo+\ee\chi', \chi_0}$ is bounded from
below on $[\chi_0].$ Since $\chi'\in [\chi_0]$, by the argument of
\cite{[Sz]}\footnote{The authors would like to thank G. Sz\'ekelyhidi for telling them this fact, which they overlooked when preparing \cite{[LSY]}.} \cite{[LSY]} there is a uniform constant $C>0$ such that for any
$\varphi\in \cH(M, \chi_0)$,
$$|\hat J_{\oo+\ee\chi', \chi_0}(\varphi)-\hat J_{\oo+\ee\chi_0, \chi_0}(\varphi)|\leq C . $$
Therefore, $\hat J_{\oo+\ee\chi_0, \chi_0}(\varphi)$ is bounded from
below and  we have \beq \hat J_{\oo, \chi_0}(\varphi)\geq -\ee
\Big(I_{\chi_0}(\varphi)-J_{\chi_0}(\varphi)\Big)-C,\quad
\forall\;\varphi\in \cH(M, \chi_0). \label{eq:011}\eeq

Now using Tian's $\al$-invariant we have (see Lemma 4.1 of
\cite{[SW]}, also \cite{[T]} page 95 ) \beqn\int_X\;\log \frac
{\chi_{\varphi}^n}{\chi_0^n}\,\frac {\chi_{\varphi}^n}{n!}&\geq&
 \al  I_{\chi_0}(\varphi)-C\nonumber\\&\geq& \frac {n+1}{n}\al \cdot (I_{\chi_0}(\varphi)
-J_{\chi_0}(\varphi))-C,\quad \forall\,\varphi\in \cH(M, \chi_0)
\label{eq:012}\eeqn for any $\al\in (0, \al_M([\chi_0])).$ Set
$\oo_0:=-Ric(\chi_0)>0$. Combining the inequalities
(\ref{eq:011})-(\ref{eq:012}) we have \beqs
\mu_{\chi_0}(\varphi)&=&\int_X\;\log \frac
{\chi_{\varphi}^n}{\chi_0^n}\,\frac {\chi_{\varphi}^n}{n!}+\hat J_{\oo_0, \chi_0}(\varphi)\\
&\geq &\int_X\;\log \frac {\chi_{\varphi}^n}{\chi_0^n}\,\frac
{\chi_{\varphi}^n}{n!}+\hat J_{\oo, \chi_0}(\varphi)-C\\&\geq &
 \Big(\frac
{n+1}{n}\al-\ee\Big)\Big(I_{\chi_0}(\varphi)-J_{\chi_0}(\varphi)\Big)-C.
\eeqs
Therefore, for sufficently small $\ee$   the $K$ energy is proper.

\end{proof}

\section{Proof of Theorem \ref{theo:Amain1}}\label{sec:Amain1}
\subsection{The modified $K$-energy  $\td \mu$ and $\td J$ functional}
We first recall some notations in \cite{[LSY]}.  Let $(M, \chi_0)$
be a $n$-dimensional compact K\"ahler manifold with a K\"ahler form
$$\chi_0=\frac {\sqrt{-1}}2h_{i\bar j}dz^i\wedge d\bar z^j.$$ We denote by
$\cH(M, \chi_0)$ the space of K\"ahler potentials
$$\cH(M, \chi_0)=\{\varphi\in C^{\infty}(M, \RR)\;|\; \chi_{\varphi}=\chi_0+\pbp \varphi>0\}.$$
A metric $\chi$ is called ``extremal" if the gradient of the scalar curvature $R(\chi)$ is a holomorphic vector field, i.e.
$$R(\chi)-\underline{R}-\te_X(\chi)=0, $$
where $\underline{R}$ is the integral mean value of $R(\chi)$ (which is a topological number) and $\te_X(\chi)$ is the normalized holomorphic potential of a holomorphic vector field $X$ with respect to the metric $\chi. $ Namely,  $\te_X(\chi)$ satisfies the equalities
$$L_X\chi=\pbp \te_X(\chi),\quad \int_M \te_X(\chi){\chi^n\over n!}=0.$$
Such a holomorphic vector field $X$ is called an ``extremal vector field".  Futaki and Mabuchi proved that ``extremal vector field" makes sense in a general K\"ahler manifold and is unique \cite{[FM]}.  Here we always assume that $L_{\Im X}\chi=0$, hence $\theta_X$ is real-valued.
For such an extremal vector field $X$, we modified the space of K\"ahler potentials accordingly:
$$\cH_X(M, \chi_0)=\{\varphi\in \cH(M,\chi_0)\;|\; \Im X(\varphi)=0\}.$$
For any $\varphi\in\cH_X(M,\chi_0)$, the potential $\theta_X(\chi_\varphi)$ is also real-valued (Since we always have $\te_X(\chi_\varphi)=\te_X(\chi)+X(\varphi)$.). Then we can define
the modified $K$-energy on $\cH_X(M,\chi_0)$ by the variational formula
$$\delta \td\mu_{\chi_0}(\varphi)=-\int_M\,\delta\varphi(R(\chi_\varphi)-\underline{R}-\te_X(\chi_\varphi))\,\frac {\chi_\varphi^n}{n!}.$$
Then the critical point of $\td \mu_{\chi_0}$ is just an extremal
K\"ahler metric in $[\chi_0]$.

 The $\td J$ functional with respect to a reference closed (1,1)-form $\oo$ (not necessarily positive) is defined by the formula
\beqn \td J_{\oo, \chi_0}(\varphi)&=&\hat J_{\oo,
\chi_0}(\varphi)+\int_0^1\,\int_M\,\pd
{\varphi_t}t\te_X(\varphi_t)\frac
{\chi_{\varphi_t}^n}{n!}dt\nonumber
\\&=&
\int_0^1\,\int_M\;\pd {\varphi_t}t (\oo\wedge
\chi_{\varphi_t}^{n-1}-c\chi_{\varphi_t}^n)\frac {
dt}{(n-1)!}+\int_0^1\,\int_M\,\pd {\varphi_t}t\te_X(\varphi_t)\frac
{\chi_{\varphi_t}^n}{n!}dt, \label{eq:003} \eeqn where
$$c={[\oo][\chi_0]^{n-1}\over [\chi_0]^n}.$$ When we choose
$\oo_0=-Ric(\chi_0)$, then a direct computation shows that
\beq\label{eqn:K-J} \td\mu_{\chi_0}(\varphi)=\int_M\;\log \frac
{\chi_{\varphi}^n}{\chi_0^n}\,\frac {\chi_{\varphi}^n}{n!}+\td
J_{\oo_0, \chi_0}(\varphi).  \eeq

Note that the modified K-energy and $J$-functional actually make sense on the larger space $\cH(M,\chi_0)$, though the value may be not real. However the $\td J$ functional enjoys the following interesting property as the usual $J$ functional:
\begin{prop}\label{prop:convexity}
When $\oo$ is positive, the real part of $\td J_{\oo, \chi_0}$ is strictly convex
along any $C^{1, 1}$ geodesics in $\cH(M, \chi_0)$, and its imaginary part is linear.
\end{prop}

\begin{proof}
Since the usual $\hat J$ functional is real valued and strictly convex along any
$C^{1, 1}$ geodesics by the work of Chen, we only need to compute
the second order derivative of the additional term. Suppose the
geodesic is $C^2$, then \beqs
{d\over dt}\int_M \dot\varphi_t \te_X(\varphi_t){\chi^n_{\varphi_t}\over n!}&=& {d\over dt}\int_M \dot\varphi_t (\te_X(\chi_0)+X(\varphi_t)){\chi^n_{\varphi_t}\over n!}\\
&=& \int_M \Big[\ddot\varphi_t \te_X(\varphi_t)+X({1\over 2}\dot \varphi^2)\Big]{\chi^n_{\varphi_t}\over n!}+\int_M \dot\varphi_t \te_X(\varphi_t)\pbp \dot\varphi_t\wedge{\chi^{n-1}_{\varphi_t}\over (n-1)!}\\
&=& \int_M \Big[\ddot\varphi_t -<\partial\dot\varphi_t,\partial\dot\varphi_t>\Big]\te_X(\varphi_t){\chi^n_{\varphi_t}\over n!}+\int_M X({1\over 2}\dot \varphi^2){\chi^n_{\varphi_t}\over n!}\\
& & -\int_M \frac {\sqrt{-1}}2 \p\te_X(\varphi_t)\wedge \bar\p ({1\over 2}\dot\varphi^2_t)\wedge{\chi^{n-1}_{\varphi_t}\over (n-1)!}\\
&=&  \int_M \Big[\ddot\varphi_t -<\partial\dot\varphi_t,\partial\dot\varphi_t>\Big]\te_X(\varphi_t){\chi^n_{\varphi_t}\over n!}+\int_M L_X\Big({1\over 2}\dot\varphi^2_t{\chi^n_{\varphi_t}\over n!}\Big)\\
&=& \int_M \Big[\ddot\varphi_t
-<\partial\dot\varphi_t,\partial\dot\varphi_t>\Big]\te_X(\varphi_t){\chi^n_{\varphi_t}\over
n!}=0.\eeqs Then we approximate a general $C^{1,1}$ geodesic by
$C^2$ geodesics as Chen-Tian \cite{[CT]}. So we conclude that $\Re\td
J$ is also strictly convex along any $C^{1, 1}$ geodesics, and $\Im\td J$ is linear.
\end{proof}

A direct corollary of Proposition \ref{prop:convexity} is the
following result:

\begin{cor}\label{lem:4.2}
If $\td J_{\oo,\chi_0}$ has a critical point $\varphi\in\cH_X(M,\chi_0)$ and  $\oo>0$, then  $\td
J_{\oo,\chi_0}$ is bounded from below on $\cH_X(M,\chi_0)$.
\end{cor}

\begin{proof}
This is a minor modification of Chen's proof of Proposition 3 in  \cite{[Chen1]}.  We just connect any $\psi\in \cH_X(M,\chi_0)$ with $\varphi$ by a $C^{1,1}$ geodesic in $\cH(M,\chi_0)$. Then since both $\td J_{\oo,\chi_0}(\varphi)$ and $\td J_{\oo,\chi_0}(\psi)$ are real valued,  by Proposition \ref{prop:convexity}, $\td J_{\oo,\chi_0}$ is real valued along this geodesic, and hence strictly convex. The rest of the proof is identical to that of Chen in \cite{[Chen1]}, so we omit it.
\end{proof}

\subsection{The existence of critical points of $\td J$}

In this subsection, we always assume $\oo$ is a closed positive
(1,1)-form, i.e. a K\"ahler form. We want to find out the critical
point of $\td J_{\oo,\chi_0}$. By definition, a critical point
$\varphi\in\cH_X(M,\chi_0)$ satisfies
 \beq \label{eq:critical}
  \oo\wedge
\chi_{\varphi}^{n-1}=\Big(c+\frac 1n\te_X(\chi_\varphi)\Big)\chi_{\varphi}^n.
\eeq
We have a similar theorem as Song-Weinkove, saying that the existence of a ``subsolution" (in a suitable sense) to the above Euler-Lagrange equation will actually leads to a solution:
\begin{theo}\label{theo:001}
If there is a metric $\chi'\in [\chi_0]$ satisfying
  \beq
(nc\chi'-(n-1)\oo)\wedge \chi'^{n-2}+\te_X(\chi')\,\chi'^{n-1}>0, \label{eq:A002}
  \eeq
and $L_{\Im\ X} \oo=0$, then there is a smooth K\"ahler metric $\chi_{\varphi}=\chi_0+\pbp
\varphi\in [\chi_0]$
  satisfying the equation (\ref{eq:critical}),  and the solution $\chi_\varphi$ is unique.
\end{theo}

Note that (\ref{eq:A002}) automatically implies $\te_X(\chi')$ is real valued. The uniqueness part of Theorem \ref{theo:001} follows directly from Proposition \ref{prop:convexity}. We only need to study the existence problem. Without loss of generality, we may assume that the initial metric $\chi_0$ satisfies (\ref{eq:A002}).
To prove Theorem \ref{theo:001}, we introduce the following flow, called ``modified J-flow":
\beqn
\pd {\varphi}t&=&c-\frac {\oo\wedge \chi_{\varphi}^{n-1}}{\chi_{\varphi}^n}+\frac 1n \Re\ \te_X(\chi_\varphi) \nonumber\\
&=&\frac 1n\Big(nc+\Re\ \te_X(\chi_\varphi)-\La_{\chi_\varphi}\oo\Big).
\label{eq:A001}
\eeqn
Denote the right hand side operator by $L(\varphi)$, then it is easy to see that the linearization of $L$ is given by $\td \Delta+{1\over n}\Re\ X$, where
$$\td \Delta f=\frac 1n h^{k\bar l}\p_k\p_{\bar l}f,\quad h^{k\bar l}=\chi^{k\bar j}
\chi^{i\bar l}g_{i\bar j}.$$
Since $\td\Delta$ is strictly elliptic, we always have short time solution to the flow equation (\ref{eq:A001}).

A modified $J$-flow starts with an element of $\cH_X(M,\chi_0)$ will remain in this space:
\begin{lem}
If $\varphi_0=\varphi|_{t=0}$ satisfies $(\Im\ X) (\varphi_0)=0$,  and $L_{\Im\ X} \oo=0$,  then along the modified $J$-flow, we always have $(\Im\ X) ( \varphi)=0$.
\end{lem}

\begin{proof}
Denote $\Im\ X$ by $Y$. By the assumption $L_Y\oo=0=L_Y\chi_0$, we have
\beqs
Y (\La_{\chi_\varphi}\oo)&=&-\chi^{\al\bar j}_{\varphi}\chi^{i\bar \bb}_{\varphi}(L_Y\chi_{\varphi})_{\al\bar \bb}g_{i\bar j}+ \chi_{\varphi}^{i\bar j}(L_Y\oo)_{i\bar j}\\
&=&-\chi^{\al\bar j}_{\varphi}\chi^{i\bar \bb}_{\varphi}\big(L_Y\chi_0+Y(\varphi)\big)_{\al\bar \bb}g_{i\bar j}\\
&=&-\chi^{\al\bar j}_{\varphi}\chi^{i\bar \bb}_{\varphi}(Y(\varphi))_{\al\bar \bb}g_{i\bar j}=-n\td\Delta Y(\varphi)
\eeqs
 Form (\ref{eq:A001}), we have
\beqs
\frac{\partial Y(\varphi)}{ \partial t}&=&-\frac 1 n Y(\La_{\chi_\varphi}\oo)+\frac 1 n Y\big(\te_X(\chi_0)+Re\ X(\varphi)\big) \\
&=& \td\Delta Y(\varphi)+ \frac 1 n (Re\ X) \big(Y(\varphi) \big)+\frac 1 n Y\big(\te_X(\chi_0)\big),
\eeqs
where the last equality follows from the fact that the real part and imaginary part of a holomorphic vector field always commute.\footnote{Note that a holomorphic vector field is always of the form $Z-iJZ$, where $Z$ is real-holomorphic, i.e. $L_Z J=0$. So we have $[Z, JZ]=L_Z(JZ)=JL_Z Z=0$.}

{\bf Claim:} We always have $Y\big(\te_X(\chi_0)\big)=0,$   thus $Y(\varphi)$ satisfies a very good parabolic equation and we conclude from maximum principle that $Y(\varphi)=0$ along the flow.

To prove the claim, just note that from the definition of $\te_X(\chi_0)$, we always have
$$X^i\chi_{i\bar j}=\partial_{\bar j}\te_X(\chi_0).$$
So we have
$$X\big(\te_X(\chi_0)\big)=X^j\bar {X^i}\chi_{j\bar i}=|X|^2_{\chi_0}. $$
Since both $X\big(\te_X(\chi_0)\big)$ and $\te_X(\chi_0)$ are real, we have $(\Im\ X) \big(\te_X(\chi_0)\big)=\Im\big( X\big(\te_X(\chi_0)\big)\big)= 0$.
\end{proof}

From the above lemma, we see that actually we can rewrite our equation as
\beqn
\pd {\varphi}t&=&c-\frac {\oo\wedge \chi_{\varphi}^{n-1}}{\chi_{\varphi}^n}+\frac 1n  \te_X(\chi_\varphi) \nonumber\\
&=&\frac 1n\Big(nc+  \te_X(\chi_\varphi)-\La_{\chi_\varphi}\oo\Big).
\label{eq:A001'}
\eeqn
Differentiating (\ref{eq:A001'}) with respect to $t$, we have
$$\pd {}{t}\pd {\varphi}t=\td \Delta \pd {\varphi}t+{1\over n}X\Big(\pd {\varphi}t\Big).$$
The maximum principle implies that
$$\min_M\;\pd {\varphi}t\Big|_{t=0}\leq \pd {\varphi}t\leq \max_M\;\pd {\varphi}t\Big|_{t=0}.$$
In particular,
$$\La_{\chi}\oo\leq \max_M \La_{\chi_0}\oo+\max_M \te_X(\chi_\varphi)-\min_M\te_X(\chi_0).$$

Since both $\chi_0$ and $\chi_\varphi$ are $\mathrm{\Im} X$-invariant
 by Zhou-Zhu \cite{[ZhZh]}, the term $\max_M
\te_X(\chi_\varphi)-\min_M\te_X(\chi_0)$ is uniformly bounded. Thus,
$\La_{\chi}\oo$ has uniform positive upper bound along the flow. In
particular, there is a uniform constant $c>0$ such that \beq
\chi_{\varphi}\geq c\;\oo \eeq  as long as the flow exists.

\begin{lem}There is a uniform constant $C>0$ such that for any $(x, t)$ we have
$$\La_{\oo}\chi\leq C e^{A(\varphi-\inf_{M\times[0, t]}\varphi)},$$
and $|\varphi|_{C^0}\leq C$ as long as the flow exists.

\end{lem}
\begin{proof}Following Song-Weinkove \cite{[SW2]}, in normal coordinates of $\oo$, we have
\beq \td \Delta (\La_{\oo}\chi)=\frac 1n h^{k\bar l}R_{k\bar l}
^{\quad  i\bar j} (g)\chi_{i\bar j}+\frac 1n h^{k\bar l}g^{i\bar
j}\p_k\p_{\bar l}\chi_{i\bar j}, \eeq where $R_{k\bar l} ^{\quad
i\bar j} (g)$ denotes the curvature tensor of $g$. By the equation
of the modified $J$-flow, we have \beqs \pd
{}t\La_{\oo}\chi&=&-\frac 1n g^{i\bar j}\p_i\p_{\bar j}(\chi^{k\bar
l}g_{ k\bar l})+
\frac 1{n } g^{i\bar j}\p_i\p_{\bar j}(\te_X(\chi_\varphi))\\
&=&\frac 1n\Big(g^{i\bar j}h^{p\bar q}\p_i\p_{\bar j}\chi_{p\bar q}-g^{i\bar
j}h^{r\bar q}\chi^{p\bar s}\p_i \chi_{r\bar s}\p_{\bar
j}\chi_{p\bar q}-g^{i\bar j}h^{p\bar s}\chi^{r\bar q}\p_i
\chi_{r\bar s}\p_{\bar j}\chi_{p\bar q}+\chi^{k\bar l}R_{k\bar
l}(g)\Big)\\&&+
\frac 1{n } g^{i\bar j}\p_i\p_{\bar j}(\te_X(\chi_\varphi)). \eeqs
Therefore, we have
\beqs &&
(\td \Delta-\pd {}t)\log(\La_{\oo}\chi)\\&=&  \frac {\td \Delta
(\La_{\oo}\chi)}{\La_{\oo}\chi}
-\frac {|\td \Na(\La_{\oo}\chi) |^2}{(\La_{\oo}\chi)^2}-\pd {}t\log(\La_{\oo}\chi)\\
&= &\frac 1{n\La_{\oo}\chi}\Big(h^{k\bar l}R_{k\bar l} ^{\quad
i\bar j}(g)\chi_{i\bar j}+g^{i\bar j}h^{r\bar q}\chi^{p\bar
s}\p_i \chi_{r\bar s}\p_{\bar j}\chi_{p\bar q}+g^{i\bar j}h^{p\bar
s}\chi^{r\bar q}\p_i \chi_{r\bar s}\p_{\bar j}\chi_{p\bar
q}\\&&-\chi^{k\bar l}R_{k\bar l}(g)-n\frac {|\td
\Na(\La_{\oo}\chi)
|^2}{ \La_{\oo}\chi }-  g^{i\bar j}\p_i\p_{\bar j}\te_X(\chi_\varphi)\Big)\\
&\geq&\frac 1{n\La_{\oo}\chi}\Big(h^{k\bar l}R_{k\bar l} ^{\quad
i\bar j}(g)\chi_{i\bar j}-\chi^{k\bar l}R_{k\bar l}(g)-  g^{i\bar j}\p_i\p_{\bar j}\te_X(\chi_\varphi)\Big),
\eeqs  where we used the inequality by Lemma 3.2 in \cite{[W1]} \beq n|\td \Na(\La_{\oo}\chi)
|^2 \leq (\La_{\oo}\chi) g^{i\bar j}h^{r\bar q}\chi^{p\bar
s}\p_i \chi_{r\bar s}  \p_{\bar j}\chi_{p\bar q}. \eeq
On the other hand, we have
\beqs
X(\log \La_{\oo}\chi)&=&\frac 1{\La_{\oo}\chi}X\Big(g^{i\bar j}(\chi_{0, i\bar j}+\varphi_{i\bar j})\Big)\\
&=&\frac 1{\La_{\oo}\chi}\Big(X(g^{i\bar j}\chi_{0, i\bar j})+X(g^{i\bar j}\varphi_{i\bar j})\Big)\\
&=& \frac 1{\La_{\oo}\chi}\Big(X(g^{i\bar j}\chi_{0, i\bar j})+\Delta_g(X(\varphi))- X^k_{, i}\varphi_{k\bar i} \Big),
\eeqs where we used the fact that
\beqs
X(g^{i\bar j}\varphi_{i\bar j}) = X^k\varphi_{i\bar ik}=X^k \varphi_{k\bar ii}=g^{i\bar j}(X(\varphi))_{i\bar j}-X^k_{, i}\varphi_{k\bar i}.
\eeqs
Combining the above identities, we have
\beqs &&
\Big(\td \Delta+\frac 1{n } X-\pd {}t\Big)\log(\La_{\oo}\chi)\\&\geq &\frac 1{n\La_{\oo}\chi}\Big(h^{k\bar l}R_{k\bar l} ^{\quad
i\bar j}(g)\chi_{i\bar j}-\chi^{k\bar l}R_{k\bar l}(g)-  g^{i\bar j}\p_i\p_{\bar j}\te_X(\varphi)+  X(g^{i\bar j}\chi_{0, i\bar j})+  \Delta_g(X(\varphi))-
  X^k_{, i}\varphi_{k\bar i}\Big)\\
&=&\frac 1{n\La_{\oo}\chi}\Big(h^{k\bar l}R_{k\bar l}^{\quad i\bar j}(g)\chi_{i\bar j}-\chi^{k\bar l}R_{k\bar l}(g) - X^k_{, i}\varphi_{k\bar i} - \Delta_g\te_X+ X(g^{i\bar j}\chi_{0, i\bar j})\Big)\\
&=& \frac 1{n\La_{\oo}\chi}\Big(h^{k\bar l}R_{k\bar l}^{\quad i\bar j}(g)\chi_{i\bar j}-\chi^{k\bar l}R_{k\bar l}(g)-   X^k_{, i}\chi_{k\bar i}+  X^k_{, i}\chi_{0, k\bar i} - \Delta_g\te_X+  X(g^{i\bar j}\chi_{0, i\bar j})\Big),
\eeqs where $\te_X$ is the holomorphic potential of $X$ with respect to $\chi_0.$
Note that
\beqs
\Big(\td \Delta+\frac 1{n } X-\pd {}t\Big)\varphi&=&\frac 1n\Big(h^{k\bar l}\varphi_{k\bar l}+\chi^{i\bar j}g_{i\bar j}-nc-\te_X\Big)\\
&=& \frac 1n\Big(2\chi^{i\bar j}g_{i\bar j}-h^{i\bar j}\chi_{0, i\bar j}-nc-\te_X\Big).
\eeqs
Thus, we have
\beqs
&&
n\Big(\td \Delta+\frac 1n X-\pd {}t\Big)\Big(\log(\La_{\oo}\chi)-A\varphi\Big)\\
&\geq& \frac 1{\La_{\oo}\chi}\Big(h^{k\bar l}R_{k\bar l}^{\quad i\bar j}(g)\chi_{i\bar j}-\chi^{k\bar l}R_{k\bar l}(g)- X^k_{, i}\chi_{k\bar i}+C(\chi_0, \oo, X)\Big)\\
&&-2A\chi^{i\bar j}g_{i\bar j}+Ah^{i\bar j}\chi_{0, i\bar j}+ncA
+A\te_X.
\eeqs By the assumption (\ref{eq:A002}), we can choose $\ee>0$ sufficiently small such that
\beq
(nc\chi_0-(n-1)\oo)\wedge \chi_0^{n-2}+\te_X(\chi_0)\,\chi_0^{n-1}>2\ee\,\chi_0^{n-1}.  \label{eq:A003}
\eeq
Moreover, since $\chi_\varphi$ is uniformly bounded from below, we can choose   $A$ large  such that
$$-\frac 1{A\La_{\oo}\chi}\Big(h^{k\bar l}R_{k\bar l}^{\quad i\bar j}(g)\chi_{i\bar j}-\chi^{k\bar l}R_{k\bar l}(g)- X^k_{, i}\chi_{k\bar i}+C(\oo, X)\Big)\leq \ee, $$
then at the maximum point $(x_0, t_0)$ of $\log(\La_{\oo}\chi)-A\varphi$, we have
\beqs
nc+\te_X+h^{i\bar j}\chi_{0, i\bar j}-2\chi^{i\bar j}g_{i\bar j}\leq
\ee.
\eeqs
We choose normal
coordinates for the metric $ \chi_{0}$ so that the metric $\chi $
is diagonal with entries $\la_1, \cdots, \la_n.$ We denote the
diagonal entries of $\oo$ by $\mu_1, \cdots, \mu_n. $ Thus, we have
$$nc+\te_X(x_0)+ \sum_{i=1}^n\frac {\mu_i}{\la_i^2}-2\sum_{i=1}^n\frac {\mu_i}{\la_i}\leq \ee, $$
which implies that for any fixed index $k$, we have the inequality \beqn \ee&\geq&\sum_{i=1, i\neq k}^n\,
\mu_i\Big(\frac 1{\la_i}-1\Big)^2-\sum_{i=1, i\neq k}^n\,
\mu_i +\frac {\mu_k}{\la_k^2}-2\frac {\mu_k}{\la_k}+nc+\te_X(x_0)\nonumber\\
&\geq & nc+\te_X(x_0)-\sum_{i=1, i\neq k}^n\, \mu_i-2\frac
{\mu_k}{\la_k}.\label{eq:mu} \eeqn

On the other hand, by (\ref{eq:A003}) we have for any $k$,
$$(nc\chi_0-(n-1)\oo)\wedge \chi_0^{n-2}\wedge\beta_k+\te_X(\chi_0)\,\chi_0^{n-1}\wedge\beta_k>2\ee\,\chi_0^{n-1}\wedge \beta_k,$$
where $\beta_k:=\sqrt{-1}dz^k\wedge d\bar z^k$. This means
$$nc+\te_X(x_0)-\sum_{i=1, i\neq k}^n\,\mu_i>2\ee. $$
Combining the above inequalities, we have
$\frac {\la_k}{\mu_k}<\frac 2{\ee}$ and there is a constant $C=C(n, \ee)$ such that  at the  point $(x_0, t_0)$,
$$\La_{\oo}\chi\leq C.$$
Thus, at any point $(x_0, t_0)$ we have the estimate
$$\La_{\oo}\chi\leq C e^{A(\varphi-\inf_{M\times[0, t]}\varphi)}.$$

The passage from this $C^2$ estimate to $C^0$ estimate does not use the equation and hence is identical to Song-Weinkove\cite{[SW]} and Weinkove \cite{[W1]}\cite{[W2]}, so we omit it.

\end{proof}

\begin{cor}
The modified J-flow exists for any $t\in [0,\infty)$.
\end{cor}

\begin{proof}
Suppose the solution exists only in $[0,T)$ with $T<\infty$. We will derive a contradiction.
By the above lemma, we have uniform $C^0$ and $C^2$ estimates. By interpolation, we also have uniform $C^1$ estimate on $[0,T)$. By Evans-Krylov estimate, we also have uniform $C^{2,\alpha}$ estimate. Then we can take limit of $\varphi(\cdot, t_i)$ as $t_i\to T$ to get a $\varphi_T$. Since $\chi_\varphi$ is uniformly bounded from below, $\chi_0+\pbp\varphi_T$ is a K\"ahler form. So the solution can extend beyond $T$, a contradiction!
\end{proof}

Now we can use the modified $J$-flow to finish the proof of Theorem
\ref{theo:001}:

\begin{proof}[Proof of Theorem \ref{theo:001}]
By our above discussion, the modified $J$-flow has a unique solution
$\varphi(\cdot,t)$ for $t\in [0,\infty)$. By the proof of the above
theorem, we also have a uniform $C^{2,\alpha}$ estimate.  By the
standard bootstrap argument, the solutions are uniformly bounded
with respect to any $C^k$ norm. Then for any sequence $t_i\to
\infty$, we can find a subsequence, also denoted by $t_i$  such that
$\varphi(\cdot, t_i)\to \varphi_\infty$ in $C^{\infty}$.  We shall
prove that $\varphi_\infty$ solves (\ref{eq:critical}).

To show this, we define an energy functional associated with the
modified $J$-functional following Chen \cite{[Chen2]}:
$$E_{X,\chi_0}(\varphi):= \int_M \big(\te_X(\chi_\varphi)-\La_{\chi_\varphi}\oo\big)^2 {\chi_\varphi^n\over n!}=\int_M \sigma^2 {\chi_\varphi^n\over n!}  ,$$
where $\sigma:= \te_X(\chi_\varphi)-\La_{\chi_\varphi}\oo$.
Then along the modified J-flow, we have (in normal coordinates of $\chi_\varphi$)
\beqs
{d\over dt} E_{X,\chi_0}(\varphi) &=& \int_M\Big[ 2\sigma \big(   X({\partial\varphi \over \partial t}) +\chi_\varphi^{i\bar q}\chi_\varphi^{p\bar j}({\partial\varphi \over \partial t})_{,p\bar q} g_{i\bar j}\big) +\sigma^2 \Delta_\chi {\partial\varphi \over \partial t} \Big]{\chi_\varphi^n\over n!} \\
&=& \int_M\Big[ {2\over n}\sigma X(\sigma)+{2\over n}\sigma \chi_\varphi^{i\bar q}\chi_\varphi^{p\bar j}\sigma_{,p\bar q} g_{i\bar j}+{1\over n}\sigma^2 \chi_\varphi^{p\bar q}\sigma_{,p\bar q}\Big]{\chi_\varphi^n\over n!} \\
&=&  \int_M\Big[ {2\over n}\sigma X(\sigma)-{2\over n} \chi_\varphi^{i\bar q}\chi_\varphi^{p\bar j}\sigma_{,\bar q}\sigma_{,p} g_{i\bar j}-{2\over n} \sigma\chi_\varphi^{i\bar q}\chi_\varphi^{p\bar j}\sigma_{,p} g_{i\bar j,\bar q}-{2\over n}\sigma \chi_\varphi^{p\bar q}\sigma_{,\bar q}\sigma_{,p}\Big]{\chi_\varphi^n\over n!} \\
&=& \int_M\Big[ {2\over n}\sigma X(\sigma)-{2\over n} \chi_\varphi^{i\bar q}\chi_\varphi^{p\bar j}\sigma_{,\bar q}\sigma_{,p} g_{i\bar j}-{2\over n} \sigma\chi_\varphi^{i\bar q}\chi_\varphi^{p\bar j}\sigma_{,p} g_{i\bar q,\bar j}-{2\over n}\sigma \chi_\varphi^{p\bar q}\sigma_{,\bar q}\sigma_{,p}\Big]{\chi_\varphi^n\over n!} \\
&=& \int_M\Big[ {2\over n}\sigma X(\sigma)-{2\over n} \chi_\varphi^{i\bar q}\chi_\varphi^{p\bar j}\sigma_{,\bar q}\sigma_{,p} g_{i\bar j}-{2\over n} \sigma\chi_\varphi^{p\bar j}\sigma_{,p}(\La_{\chi_\varphi}\oo)_{,\bar j} -{2\over n}\sigma \chi_\varphi^{p\bar q}\sigma_{,\bar q}\sigma_{,p}\Big]{\chi_\varphi^n\over n!} \\
&=& \int_M\Big[ {2\over n}\sigma X(\sigma)-{2\over n} \chi_\varphi^{i\bar q}\chi_\varphi^{p\bar j}\sigma_{,\bar q}\sigma_{,p} g_{i\bar j}-{2\over n}\sigma \chi_\varphi^{p\bar q}\big(\te_X(\chi_\varphi)\big)_{,\bar q}\sigma_{,p}\Big]{\chi_\varphi^n\over n!} \\
&=& -{2\over n} \int_M \chi_\varphi^{i\bar q}\chi_\varphi^{p\bar j}\sigma_{,\bar q}\sigma_{,p} g_{i\bar j}{\chi_\varphi^n\over n!}.
\eeqs
The last equality comes from the definition of $\te_X(\chi_\varphi)$. So we have
$${d\over dt} E_{X,\chi_0}(\varphi)=-{2\over n} \int_M |\nabla^\chi \sigma|^2_\oo {\chi_\varphi^n\over n!}<0.$$
In particular, this implies
$$\int_0^\infty \Big( \int_M |\nabla^{\chi_t} \sigma(\cdot, t)|^2_\oo {\chi_t^n\over n!} \Big)dt<\infty.$$
So if the sequence $t_i$ is chosen properly, so that
$$\int_M |\nabla^{\chi_{t_i}} \sigma(\cdot, t_i)|^2_\oo {\chi_{t_i}^n\over n!}\to 0 ,$$
then we can conclude that
$$\te_X(\chi_{\varphi_\infty})-\La_{\chi_{\varphi_\infty}}\oo\equiv const.$$
This implies that $\varphi_\infty$ solves (\ref{eq:critical}). The
theorem is proved.
\end{proof}

\subsection{Proof of Theorem \ref{theo:Amain1}}

\begin{proof}[Proof of Theorem \ref{theo:Amain1}]
We focus on the $G=\{1\}$ case, the proof in the general case is
identical. Recall the Aubin-Yau functionals \beqs
I_{\chi_0}(\varphi)&=&\int_X\;\varphi(\frac {\chi_0^n}{n!}-\frac {\chi_{\varphi}^n}{n!}),\\
J_{\chi_0}(\varphi)&=&\int_0^1dt\int_X\;\pd
{\varphi_t}t(\frac {\chi_0^n}{n!}-\frac {\chi_{\varphi}^n}{n!}).
\eeqs Direct calculation shows that
\beqs
I_{\chi_0}(\varphi)-J_{\chi_0}(\varphi)&=&-\int_0^1\,dt\int_X\;\pd
{\varphi}t \Delta_{\chi_{\varphi}}\varphi\,\frac {\chi_{\varphi}^n}{n!}\\
&=&-\int_0^1\,\int_X\;\pd {\varphi}t (\chi_{\varphi}^n-\chi_0\wedge
\chi_{\varphi}^{n-1})\frac {dt}{(n-1)!}. \eeqs

As $\hat J$, we also have the following lemma for $\td J$, whose proof is the same as in \cite{[Sz]} and \cite{[LSY]}:

\begin{lem}\label{lem:4.1}
If $\td J_{\oo,\chi}$ is bounded from below, then so is $\td
J_{\oo',\chi}$ for any $\oo'\in[\oo]$. ($\oo'$ may not be positive.)
\end{lem}

Set
$$c:=\frac {\big(\epsilon[\chi_0]- \pi c_1(M)\big)\cdot
[\chi_0]^{n-1}}{[\chi_0]^n}=-\frac { \pi c_1(M)\cdot
[\chi_0]^{n-1}}{[\chi_0]^n}+\epsilon.$$
By condition (2) of Theorem \ref{theo:Amain1}, we can find a $\oo\in \epsilon[\chi_0]- \pi c_1(M)$ with $\oo>0$ (Since $\min\theta_X<0$, $\epsilon[\chi_0]- \pi c_1(M)>0$. ). Since $\Im X$ generates a compact one-parameter group of holomorphic automorphisms, we can average $\oo$. So we can also assume that $L_{\Im X}\oo=0$. Then by condition (3), we can find a closed (1,1)-form $\chi'\in[\chi_0]$ such that
$$(nc+\min\theta_X)\chi'-(n-1)\oo>0.$$

We claim that $nc+\min\theta_X>0$, thus the above inequality also
implies $\chi'>0$. In fact, by condition (2), we have
$$\frac{ \pi c_1(M)\cdot [\chi_0]^{n-1}}{[\chi_0]^n}<\ee+\min\theta_X.$$
So, $c>-\min\theta_X$, and hence
$$nc+\min\theta_X>-(n-1)\min\theta_X>0.$$

Now we have
$$(nc+\theta_X(\chi'))\chi'-(n-1)\oo\geq (nc+\min\theta_X)\chi'-(n-1)\oo>0.$$
By Theorem \ref{theo:001}, we know that $\td J_{\oo,\chi_0}$ has
critical point. By Lemma \ref{lem:4.1} and \ref{lem:4.2}, we
conclude that (Remember that $\oo_0=-Ric(\chi_0)$)
$$\td J_{\oo_0+\ee\chi_0,\chi_0}\geq -C.$$
By (\ref{eqn:K-J}), we have \beqs \td\mu_{\chi_0}(\varphi) &=&
\int_M\;\log \frac {\chi_{\varphi}^n}{\chi_0^n}\,\frac
{\chi_{\varphi}^n}{n!}+\td J_{\oo_0,
\chi_0}(\varphi)\\
&=& \int_M\;\log \frac {\chi_{\varphi}^n}{\chi_0^n}\,\frac
{\chi_{\varphi}^n}{n!}+\td J_{\oo_0+\ee\chi_0,
\chi_0}(\varphi)-\ee(I_{\chi_0}-J_{\chi_0})(\varphi)\\
& \geq & \big( {n+1\over n}\alpha
-\ee\big)(I_{\chi_0}-J_{\chi_0})(\varphi)-C, \eeqs for any positive
$\alpha<\alpha_M([\chi_0])$.  By condition (1), we can choose such
an  $\alpha<\alpha_M([\chi_0])$ with $ {n+1\over n}\alpha -\ee>0$,
so the modified $K$-energy is proper.
\end{proof}

\end{document}